\documentclass[11pt]{article} 
\usepackage[dvipdfmx]{graphicx}
\usepackage{tikz}
\usepackage{amsmath}
\usepackage{amssymb}
\usepackage{amsthm}
\usepackage{amssymb}
\usepackage{verbatim}

\title{Some extensions of Delete Nim}
\author{Tomoaki Abuku\thanks{National Institute of Informatics, buku3416@gmail.com}
\and
Ko Sakai\thanks{Kanagawa University, gummosakai@gmail.com} 
\and
Masato Shinoda\thanks{Nara Women's University, shinoda@cc.nara-wu.ac.jp} 
\and
Koki Suetsugu\thanks{National Institute of Informatics, suetsugu.koki@gmail.com}
}

\date{2022}

\begin{document}
\theoremstyle{definition} 
\newtheorem{theorem}{Theorem}[section]
\newtheorem{definition}[theorem]{Definition}
\newtheorem{lemma}[theorem]{Lemma}
\newtheorem{proposition}[theorem]{Proposition}
\newtheorem{corollary}[theorem]{Corollary}
\newtheorem{example}[theorem]{Example}
\newtheorem{conjecture}[theorem]{Conjecture}

\maketitle

\begin{abstract}
Nim is a well-known combinatorial game with several variants, e.g., Delete Nim and Variant Delete Nim.
In Variant Delete Nim, the player deletes one of the two heaps of stones and splits the other heap on his/her turn.
In this paper, we discuss generalized Variant Delete Nim, which generalizes the number of stone heaps to three or more, All-but-one-delete Nim, Half-delete Nim, No-more-than-half-delete Nim, and Single-delete Nim. We study the win-loss conditions for each of these games.

\end{abstract}
\section{Introduction}
Combinatorial games are $2$-player games with neither chance elements nor hidden information (for the details on combinatorial game theory, see, e.g., Albert et al. \cite{alb19} and Siegel \cite{Siegel}). 

Nim is one of the oldest and most well-known combinatorial games. The basic rule of the game is that two players take turns choosing one of several heaps and take as many tokens from the heap as they like, and the player who cannot take a token from the heap loses.
In this paper, we only deal with the rule that a player loses if he/she cannot make any possible moves on his/her turn. This rules is called the normal rule.

Nim is a two-player zero-sum complete information-confirmation finite game.
Since there are no draws, every position can be classified into two types: $\mathcal{N}$-position, in which the player whose turn it is to play has a strategy to win the game, and $\mathcal{P}$-position, in which the player whose turn it is to play does not have a winning strategy. The distinction mentioned above between $\mathcal{N}$-positions and $\mathcal{P}$-positions in Nim was shown by Bouton \cite{bou02}, and it is now known that the win-loss conditions of this game contain mathematically interesting structures. In addition, a more detailed analysis of the game has been conducted to obtain the Sprague-Grundy values for each game. The Sprague-Grundy values are useful for determining the winners of games that combine multiple games. For more information on these games, refer to \cite{alb19}.

Variants of the game with different rules for taking tokens in Nim
Moore's game, Welter-Sato's game (Maya game), Wythoff's game, etc., have been proposed.
The mathematical analysis of these games is a subject of interest. There are also games, such as Grundy's game, for splitting heaps of tokens.

In this paper, we propose a generalization of Delete Nim and discuss its win-loss conditions. 

This paper is organized as follows. Section 2 describes the rules of Delete Nim; the content of this section is concerned with the case where the number of stone heaps is two. Section 3 and beyond describe the various rules for extending the number of stone heaps in Delete Nim to three or more and the conditions for determining the winners.
Section 3 deals with All-but-one-delete Nim, Section 4 with No-more-than-half-delete Nim, Section 5 with Half-delete Nim, and Section 6 with Single-delete Nim.

Note that there is a game with a rule called ``Split-and-delete Nim" that reverses the order of deleting heaps and splitting heaps (see Abuku et al. \cite{abu22}).
In order to clearly distinguish between these rules, we call the rule defined in this paper ``Delete-and-split Nim".

\section{Delete Nim}

%
%
%
%
%
%
In this section, we review the rules of Delete Nim and the determination of winners and losers as previously mentioned in Abuku and Suetsugu \cite{abu21}.

\begin{definition}[The rules of Delete Nim]
There are $2$ heaps of tokens. The player performs the following two operations in succession on his/her turn.
\begin{itemize}
    \item Selects a non-empty heap and deletes the other heap.
    \item Removes $1$ token from the selected heap and splits the heap into two (possibly empty heaps).
\end{itemize}
\end{definition}


The $p$-adic valuation of an integer $n$, which we shall denote by $v_p(n)$, is the exponent of the highest power of the prime number $p$ that divides $n$. In this paper, only $2$-adic valuation will be treated.  

\begin{theorem}[Abuku and Suetsugu \cite{abu21}]\label{DN}
Let $\langle x,y \rangle$ be a Delete Nim position, where $x$ and $y$ represent the number of tokens in each heap. The Sprague-Grundy value of $\langle x,y \rangle$ is $v_2((x\vee y)+1)$, 
where $\vee$ is the bitwise OR-operation. In particular, $\langle x,y \rangle$ is a  $\mathcal{P}$-position if and only if both $x$ and $y$ are even.
\end{theorem}

The following game, called Variant Delete Nim or VDN, is defined in Stankova and Rike \cite{sta08}.

\begin{definition}[The rules of VDN]
There are $2$ heaps of tokens. The player performs the following two operations in succession on his/her turn.
\begin{itemize}
    \item Selects one heap and deletes it.
    \item Splits the remaining heap into $2$ (non-empty) heaps.
\end{itemize}
\end{definition}


VDN and Delete Nim are equivalent games with respect to legal moves.
We can see that the position $\langle x,y \rangle$ of VDN corresponds to the position $\langle x-1,y-1\rangle$ of Delete Nim.
Therefore, according to Theorem \ref{DN}, in VDN, the condition that each member of the pair
 $\langle x,y \rangle$ be odd is required for $\langle x,y \rangle$ to be a $\mathcal{P}$-position.
This decision condition is also included in Theorem \ref{thm_abo} of this paper. 
%

In the following sections, we will consider various rules when the number of token heaps $n$ is generalized to 3 or more.
The extended rules discussed below all follow the VDN setting, where some of the $n$ heaps of tokens are deleted and the remaining heaps are split, and the number of heaps $n$ before and after the turn is assumed to remain the same. Also, when splitting a heap of tokens every heap must contain at least one token.

The two operations that a player performs in a turn (deleting and splitting a heap) are called a {\it move}, and a position that can be transitioned from one position to another by a single move is called an {\it option}.
A position that has no option is called a {\it terminal position}, and the player whose turn it is to play in this terminal position loses the game.

\section{All-but-one-delete Nim}
In this section, we introduce a variant of delete nim, All-but-one-delete Nim or ABO-delete Nim. In this ruleset, all heaps except for one heap are removed in a move.

\subsection{The rule of ABO-delete Nim}
\begin{definition}[ABO-delete Nim]
\label{def_ABOdelnim}
There are $n$ heaps of tokens. The player performs the following two operations in succession on his/her turn.
\begin{itemize}
    \item Selects $n-1$ heaps and deletes them.
    \item Splits the remaining $1$ heap into $n$ heaps.
\end{itemize}
\end{definition}

The set of all positions in ABO-delete Nim is $G_n = \{ \langle z_1, z_2, \ldots, z_n \rangle \mid  z_1, z_2, \ldots, z_n \in \mathbb{N}\}$.
If $n = 2$, then the ruleset is the same as VDN. Therefore, ABO-delete Nim can be considered as a generalization of VDN.
From Definition \ref{def_ABOdelnim}, the set of terminal positions of ABO-delete Nim is $\{ \langle z_1, z_2, \ldots, z_n \rangle \mid 1 \leq z_1, z_2, \ldots, z_n \leq n-1\}$.

\subsection{Characterizing positions of ABO-delete Nim}
\begin{theorem}
\label{thm_abo}
All-but-one-delete Nim position $\langle z_1, z_2,\ldots, z_n \rangle$ is a $\mathcal{P}$-position if and only if 

{\rm (*)} for every $i$, the remainder of $z_i$ divided by $n(n-1)$ is between $1$ and $n-1$.
\end{theorem}

Note that this theorem generalizes Theorem \ref{DN}.
\begin{proof}
Let $G_n = \{ \langle z_1, z_2, \ldots, z_n \rangle \mid  z_1, z_2, \ldots, z_n \in \mathbb{N}\}$, $P$ be the subset of $G_n,$ which satisfies (*), and $N = G_n \setminus P$. Obviously, $P$ contains the terminal positions of ABO-delete Nim and this game is not a loopy game. Therefore, it is enough to show (i) every position in $P$ has no option in $P$ and (ii) every position in $N$ has at least one option in $P$.

(i)
Assume that $Z = \langle z_1, z_2, \ldots, z_n \rangle \in P$ and $Z' = \langle z'_1, z'_2, \ldots, z'_n \rangle$ is an option of $Z$. Then there exists $i$ such that $z_i = z'_1 + z'_2 + \cdots + z'_n$. Therefore, from (*), the reminder of $z'_1 + z'_2 + \cdots + z'_n$ divided by $n(n-1)$ is larger than $0$ and less than $n$. On the other hand, if $z'_1, z'_2, \ldots, z'_n$ satisfies (*), then the sum of reminders of $z'_1, z'_2, \ldots, z'_n$ is between $n$ and $n(n-1)$, which is a contradiction.
Thus, $Z'  \not \in P$.

(ii)
Assume that $Z = \langle z_1, z_2, \ldots, z_n \rangle \in N$. Then, there exists $z_i$ whose reminder divided by $n(n-1)$ is larger than or equal to $n$ (Here, we say the reminder is $n(n-1)$ if $z_i$ can be divided by $n(n-1)$).
By removing all heaps except for this heap and splitting this heap into $n$ heaps, one can have a position in $P$.
\end{proof}

How to split the heap into $n$ heaps at the last of the proof will be presented in Lemma \ref{lem_oddoidevenoid}(2).

\section{No-more-than-half-delete Nim}
Next, we introduce No-more-than-half-delete Nim or NMTH-delete Nim. In this ruleset, the player removes no more than half heaps of all heaps in a move.

\subsection{The rule of NMTH-delete Nim}
\begin{definition}[NMTH-delete Nim]
There are $n$ heaps of tokens. The player performs the following two operations in succession on his/her turn.
\begin{itemize}
    \item Chooses a positive integer $k$ such that $k\leq \frac{n}{2}$, selects $k$ heaps, and deletes them.
    \item Selects $k$ heaps of the remaining $n-k$ heaps and splits each heap into two heaps.
\end{itemize}
\end{definition}

Note that if $n = 2$, then the rule is the same as VDN and if $n = 3,$ then the rule is the same as single delete nim with the number of heaps is $3$, presented in Section \ref{sec_single}.

\subsection{Characterizing positions in NMTH-delete Nim}

We found following theorem for NMTH-delete Nim.

\begin{theorem}
\label{thm_nmth}
No-more-than-half-delete Nim position $\langle z_1, z_2, \ldots, z_n \rangle$ is a $\mathcal{P}$-position if and only if $z_1, z_2, \ldots, z_n$ satisfy the following condition:
\renewcommand{\labelenumi}{(\roman{enumi})}
\begin{enumerate}
\item $v_2(z_1)=v_2(z_2)=\cdots=v_2(z_n)=0$ if $n$ is even,
\item $v_2(z_1)=v_2(z_2)=\cdots=v_2(z_n)$ if $n$ is odd.
\end{enumerate}
\end{theorem}

For proving this theorem, we prepare following propositions. They are trivial, so we omit the proof.

\begin{proposition}
\label{prop_2adic1}
For $x,y,z \in \mathbb{N}$, if $x + y = z$, then following (i) and (ii) holds.
\renewcommand{\labelenumi}{(\roman{enumi})}
\begin{enumerate}
\item If $v_2(x) = v_2(y),$ then $v_2(z) > v_2(x)$.
\item If $v_2(x) \neq v_2(y),$ then $v_2(z) = \min \{ v_2(x), v_2(y)\}.$
\end{enumerate}
\end{proposition}

\begin{proposition}
\label{prop_2adic2}
Assume that $z \in \mathbb{N}$ and $v_2(z) > 0$. For any nonnegative integer $k < v_2(z)$, there exist $x, y \in \mathbb{N}$ such that $x+y=z$ and $v_2(x) = v_2(y) =k$.
\end{proposition}

For instance, $(x,y)= (z - 2^k,2^k)$ satisfies $x+y=z$ and $v_2(x) = v_2(y) =k$.

In the rest of this paper, we say a heap is {\em even (resp. odd) heap} if the number of stones of the heap is an even (resp. odd) number.

\begin{proof}[Proof of Theorem \ref{thm_nmth}]
Let $P_e$ be the set of positions in which every heap is an odd heap and $N_e$ be the set of positions in which at least one heap is an even heap.
We also let $P_o$ be the set of positions such that every $2$-adic valuation of the size of a heap is the same number and $N_o$ be the set of positions with several different $2$-adic valuations of the size of a heap.

(i) Assume that $n$ is an even number.
Since there are one even heap and one odd heap after an odd heap is split, a position in $P_e$ has no option in $P_e$. Next, consider a position in $N_e$. If the number of even heaps is larger than or equal to $\frac{n}{2}$, then by deleting heaps other than $\frac{n}{2}$ even heaps and splitting the remaining even heaps into odd heaps, one can obtain a position in $P_e$. If there are less than $\frac{n}{2}$ even heaps, then by deleting the same number of odd heaps as even heaps and splitting all even heaps into odd heaps, one can obtain a position in $P_e$.

(ii) Assume that $n$ is an odd number.
Then, after one move, at least one heap remains undeleted and unsplit.
From the contraposition of proposition \ref{prop_2adic1}, if one heap is split, then at least one $2$-adic valuation of the heaps is differ from $2$-adic valuation of the original heap. Therefore, every position in $P_o$ has no option in $P_o$.
Next, in position $\langle z_1, z_2, \ldots, z_n \rangle \in N_o,$ assume that $v_2(z_1) \leq v_2(z_2) \leq \cdots \leq v_2(z_n)$ and $v_2(z_1) \neq v_2(z_n)$ and we show that from this position one can obtain a position in $P_o$ in a single move.
If the number of $j$ such that $v_2(z_1) < v_2(z_j)$ is less than or equal to $\frac{n-1}{2},$ then from Proposition \ref{prop_2adic2}, one can split every heap whose $2$-adic valuation is larger than $v_2(z_1)$ into two heaps whose $2$-adic valuations are $v_2(z_1)$.
For the other cases, one can split $\frac{n-1}{2}$ heaps whose $2$-adic valuations are larger than $v_2(z_1)$ into $n-1$ heaps whose $2$-adic valuations are $v_2(z_1)$.
\end{proof}

\section{Half-delete Nim}
In this section, we introduce Half-delete Nim.
In this ruleset, differ from NMTH-delete nim, the player removes just half heaps of all heaps in a move, so the number of heaps in this ruleset must be an even number.
We also introduce a generalization of this ruleset.

\subsection{The rule of Half-delete Nim}

\begin{definition}[Half-delete Nim]
There are $n$ ($=2m$) heaps of tokens. The player performs the following two operations in succession on his/her turn.
\begin{itemize}
    \item Selects $m$ heaps and deletes them.
    \item Splits each of the remaining $m$ heaps into two heaps.
    \end{itemize}
\end{definition}
In particular, if $n=2$, this game is the same as VDN.

\subsection{Characterizing positions in Half-delete Nim}
\begin{theorem}
\label{thm_half}
Let $Z=\langle z_1, z_2, \ldots, z_{2m} \rangle$ be a Half-delete Nim position, where each $z_i$ is the number of tokens and $z_i \leq z_{i+1}$ for any $i$. 
Let $2^s$ be the smallest power of $2$ greater than $z_{m+1}$.  
Then $Z$ is a $\mathcal{P}$-position if and only if $z_1, z_2, \ldots, z_{2m}$ satisfy both of the following two conditions:

\renewcommand{\labelenumi}{(\alph{enumi})}
\begin{enumerate}
\item
all $z_1, z_2, \ldots, z_{m+1}$ are odd,
\item
For any $l$, if $z_l$ is even, then $2^s \leq z_l$.
\end{enumerate}
\end{theorem}
This theorem is a special case of Theorem \ref{thm_main} in the next subsection, so the proof is omitted.

\subsection{$\frac{k-1}{k}n$-delete nim}
We consider a generalization of Half-delete Nim.

\begin{definition}[$\frac{k-1}{k}n$-delete Nim]
There are $n$ ($=km$) heaps of tokens. The player performs the following two operations in succession on his/her turn.
\begin{itemize}
    \item Selects $(k-1)m$ heaps and deletes them.
    \item Splits each of the remaining $m$ heaps into $k$ heaps.
    \end{itemize}
\end{definition}
In particular, if $k=2$, this game is the same as Half-delete Nim, and if $k=n$, this game is the same as ABO-delete Nim.

\begin{definition}
A positive integer whose remainder divided by $k(k-1)$ lies between $1$ and $k-1$ is called a $k$-{\em oddoid number}, and any other positive integer is called a $k$-{\em evenoid number}.
A heap with an oddoid number of tokens is called a $k$-{\em oddoid heap}, and a heap with an evenoid number of tokens is called a $k$-{\em evenoid heap}.
\end{definition}

In particular, if $k=2$, oddoid and evenoid numbers are consistent with the usual notion of odd and even numbers.

\begin{lemma}\label{lem_oddoidevenoid}
\indent
\renewcommand{\labelenumi}{(\arabic{enumi})}
\begin{enumerate}
\item It is not possible to split an $k$-oddoid number into $k$ $k$-oddoid numbers.

\item All integers $x$ between $k$ and $k(k-1)$ can be split into $k$ integers that are between $1$ and $k-1$.

\item Let $s$ be a positive integer. Every $k$-evenoid number $y < k^s$ can be split into $k$ $k$-oddoid numbers which are less than $k^{s-1}$.
\end{enumerate}
\end{lemma}

\begin{proof}
(1) We can prove this in the similar way to (i) in the proof of Theorem \ref{thm_abo}. That is, if we can split a $k$-oddoid number into $k$ $k$-oddoid number, then we have a contradiction because the sum of reminders of all $k$-oddoid numbers split by $k(k-1)$ is between $k$ and $k(k-1)$, which contradicts to the original number is an $k$-oddoid number.

(2) Let $x = kp + q \ (0\leq q \leq k-1)$. Then, $1 \leq p \leq k-1$.
If $p < k-1$, then $x = q(p+1) + (k-q)p$ and if $p = k-1$, then $x = kp$. Thus, for both cases, $x$ can be split into $k$ numbers which are between $1$ and $k-1$.

(3) Since $k^s - k$ can be divided by $k(k-1)$, the reminder of $k^s$ divided by $k(k-1)$ is $k$. Thus, the largest $k$-evenoid number less than $k^s$ is $k^s - k$.
Therefore, for the case $s \leq 2$, we have proved in (2).

Assume that $s \geq 3$.
$$
k^s - k = \frac{k(k^{s-2}-1)}{k-1} k(k-1) + k(k-1)
$$
and $\frac{k(k^{s-2}-1)}{k-1}$ is an integer, so for a $k$-evenoid number $y < k^s$,

$$
y = \alpha k(k-1) + \beta \ \ \left(\alpha \leq \frac{k(k^{s-2}-1)}{k-1}, k\leq \beta \leq k(k-1)\right)
$$
and we can split $\alpha$ into $\alpha_1, \alpha_2, \ldots, \alpha_k \leq \frac{ k^{s-2}-1}{k-1}$. From (2), we can also split $\beta$ into $1 \leq \beta_1, \beta_2, \ldots, \beta_k \leq k-1$.
Let $\gamma_i = \alpha_i k(k-1)+\beta_i$ for any $1\leq i\leq k$,then $\gamma_1, \gamma_2, \ldots, \gamma_k$ are all $k$-oddoid numbers, $\gamma_1+\gamma_2+\cdots + \gamma_k = y$, and $\gamma_i \leq \frac{k^{s-2}-1}{k-1} k(k-1)+k-1<k^{s-1}.$
\end{proof}

Using this lemma, we can give the following winning strategy for $\frac{k-1}{k} n$-delete Nim by replacing the odd and even heaps of Theorem \ref{thm_half} with $k$-oddoid and $k$-evenoid heaps, respectively.

\begin{theorem}\label{thm_main}
Let $Z=\langle z_1, z_2, \ldots, z_{km}\rangle$ be the $\frac{k-1}{k} n$-delete nim position, where each $z_i$ is the number of tokens and $z_i < z_{i+1}$.
Let $k^s$ be the smallest power of $k$ greater than $z_{(k-1)m+1}$.
Then $Z$ is a $\mathcal{P}$-position if and only if $z_1, z_2, \ldots, z_{km}$ satisfy both of the following two conditions:

\begin{enumerate}
\item[(a)]
all $z_1, z_2, \ldots, z_{(k-1)m+1}$ are $k$-oddoid,

\item[(b)]
For any $l$, if $z_l$ is $k$-evenoid, then $k^s \leq z_l$.
\end{enumerate}
\end{theorem}

\begin{proof}
Let $P$ be the set of positions which satisfy both (a) and (b), and $N$ be the complement of $P$.
We show a position in $P$ has no option in $P$ in (i), and a position in $N$ has at least one option in $P$ in (ii) and (iii).

(i)
Assume that in a move of $\frac{k-1}{k}n$-delete Nim, the remaining all $m$ heaps are $k$-oddoid heaps.
Then from Lemma \ref{lem_oddoidevenoid} (1), one can obtain at most $k-1$ $k$-oddoid heaps by splitting a remaining heap. Thus, each option is not in $P$.
Therefore, if an option of a position in $P$ is also in $P$, a $k$-evenoid heap has to be split into $k$ $k$-oddoid heaps.
At least one $k$-oddoid heap after this split has more than $k^{s-1}$ stones.
On the other hand, the player has to split at least one of the heaps whose sizes are $z_1, z_2, \ldots, z_{(k-1)m+1}$, but any $k$-evenoid heap from this split has less than $k^s$ stones, which contradicts to (b).

(ii)
Assume that (a) is not satisfied. That is, there exists a $k$-evenoid $z_i\ (1\leq i\leq (k-1)m+1)$.
Let $s$ be an integer such that $k^{s-1} \leq z_i < k^s$.
Since $z_i$ is a $k$-evenoid number, $s \geq 2$.
Consider to split $i$-th heap and $(k-1)m+2, (k-1)m+3, \ldots, km$-th heaps.
From Lemma \ref{lem_oddoidevenoid}(3), $z_i$ can be split into $k$ $k$-oddoid heaps less than $k^{s-1}$.
If $z_j(j\geq (k-1)m+2)$ is a $k$-evenoid number, then let $z_j = \alpha k(k-1)+\beta(k\leq \beta \leq k(k-1))$. Then, from Lemma \ref{lem_oddoidevenoid}(2), $z_j$ can be split into $\beta_1, \beta_2, \ldots, \beta_{k-1}, \alpha k (k-1)+\beta_k,$ where $1 \leq \beta_1, \beta_2, \ldots, \beta_k \leq k$.
If $z_j(j\geq (k-1)m+2)$ is a $k$-oddoid number, then $z_j$ can be split into $1,1,\ldots, 1,z_j-(k-1)$.
Here, all $k^{s-1}, k^{s-1}+1, \ldots, k^{s-1}+k-2$ are $k$-evenoid number, so if $k^{s-1}\leq z_i<z_j$ and $z_j$ is a $k$-evenoid number, then $k^{s-1} \leq z_j - (k-1)$.
Therefore, for a position, if (a) is not satisfied, then the position has an option in $P$.

(iii)
Consider the case that (a) is satisfied but (b) is not satisfied.
That is, $z_1, z_2, \ldots, z_{(k-1)m+1}$ are $k$-oddoid numbers and there exists $z_i(i > (k-1)m+1)$ such that $z_i$ is a $k$-evenoid number and $z_i < k^s$.
Split the heaps whose sizes are $z_{(k-1)m+1}, z_{(k-1)m+2}, \ldots, z_{km}$ as follows:
From Lemma \ref{lem_oddoidevenoid}(3), $z_i$ can be split into $k$ $k$-oddoid numbers less than $k^{s-1}$.
For other $z_j$, similar to (ii), if $z_j$ is a $k$-evenoid number, then it can be split into $\beta_1, \beta_2, \ldots, \beta_{k-1}, \alpha k(k-1)+\beta_k$ and if $z_j$ is a $k$-oddoid number, then it can be split into $1,1, \ldots, 1,z_j-(k-1)$.
Thus,  for a position, if (a) is satisfied but (b) is not satisfied, then the position has an option in $P$.
\end{proof}

\section{Single-delete Nim}
Finally, in this section, we consider Single-delete Nim.
 In this ruleset, the player can remove only one heap.

\label{sec_single}
\subsection{The rule of Single-delete Nim}
\begin{definition}[Single-delete Nim]
There are $n$ heaps of tokens. The player performs the following two operations in succession on his/her turn.
\begin{itemize}
    \item Selects one heap and deletes it.
    \item Selects one heap of the remaining $n-1$ heaps and splits it into two heaps.
\end{itemize}
\end{definition}

If $n = 2$, then the ruleset is the same as VDN, so this ruleset is a generalization of VDN.
The terminal position in Single-delete Nim is only $\langle 1, 1, \ldots, 1 \rangle$.


\subsection{Characterizing positions in Single-delete Nim}
\begin{theorem}
If $n=3$ in the Single-delete Nim,
the position $\langle x,y,z \rangle$ is a $\mathcal{P}$-position if and only if
$v_2(x) = v_2(y) = v_2(z)$.
\end{theorem}

This result was introduced in Sakai \cite{sak21}. This theorem is a special case of Theorem \ref{thm_nmth}.

Further, we introduce a theorem for the case $n = 4$.

\begin{theorem}
Denote by $I_k(z)$ the $k$-th digit from the bottom of the binary representation of non-negative integer $z$.
For $n=4$ in the Single-delete Nim position $\langle w, x, y, z \rangle$,
let $a = v_2(w)$, $b = v_2(x)$, $c = v_2(y)$, $d = v_2(z)$.
If $a\leq b\leq c \leq d$, $\langle w, x, y, z\rangle$ is a $\mathcal{P}$-position if and only if $a,b,c$,  and $d$ satisfy one of the following conditions (1), (2), (3), (4), or (5).

\begin{itemize}
    \item [(1)]$a = b = c = d$.
    \item[(2)]$a < b = c = d$ and 
    \begin{itemize}
    \item [(2A)]$I_{d+1}(w) = 0$.
   \end{itemize}
    \item[(3)]$a < b < c = d$ and the following conditions (3A)-(3C) are satisfied. 
    \begin{itemize}
    \item [(3A)]$I_{d+1}(w) = I_{d+1}(x) = 0$.
    \item[(3B)]$I_k(w) + I_k(x) \geq 1$ for $b + 2 \leq k \leq d$.
    \item [(3C)] $I_{b+1}(w) = 1$.
   \end{itemize}
   \item[(4)]$a < b < c < d$ and the following conditions (4A)-(4E) are satisfied.
   \begin{itemize}
    \item [(4A)]$I_{d+1}(w) = I_{d+1}(x) = I_{d+1}(y) = 0$.
    \item[(4B)]$I_j (w)+I_j (x)+I_j (y) \geq 2$ for $c+2 \leq j \leq d$.
    \item [(4C)] $I_{c+1}(w) = I_{c+1}(x) = 1$.
    \item [(4D)]$I_k(w) + I_k(x)\geq  1$ for $b + 2 \leq k\leq c$.
    \item [(4E)]$I_{b+1}(w) = 1$.
\end{itemize}

\item[(5) ]$a < b < c < d$ and the following conditions (5A)-(5F) are satisfied.
   \begin{itemize}
    \item [(5A)]$I_i(w) + I_i(x) + I_i(y) + I_i(z) \in \{0, 3, 4\}$ for $i \geq d + 2$.
    \item[(5B)]$I_{d+1}(w) = I_{d+1}(x) = I_{d+1}(y) = 1$.
    \item [(5C)]$I_j (w)+I_j (x)+I_j (y) \geq 2$ for $c+2 \leq j \leq d$.
    \item [(5D)]$I_{c+1}(w) = I_{c+1}(x) = 1$.
    \item [(5E)]$I_k(w) + I_k(x) \geq 1$ for $b + 2 \leq k \leq c$.
    \item [(5F)]$I_{b+1}(w) = 1$.
\end{itemize}
   
\end{itemize}

\end{theorem}

A proof of this theorem is shown in a Japanese report \cite{shi22}.
This theorem solves only the case of $n=4$, and the proof is long and complex, so we omit it.

\section*{Acknowledgements}
This work was partially supported by
JSPS KAKENHI Grant Numbers JP21K12191 and
JP22K13953.

%
%
%
%

\end{document}